\documentclass{amsart}
\usepackage[T1]{fontenc}
\usepackage[utf8]{inputenc}
\usepackage{mathtools}
\usepackage{amsmath}
\usepackage{amssymb}
\usepackage{upref}
\usepackage{cite}
\usepackage{comment}
  \newcommand{\cref}[1]{\ref{#1}}

\newenvironment{acknowledgements}{\paragraph{\bf Acknowledgements.}}{}

\theoremstyle{definition}
\newtheorem{definition}{Definition}
\newtheorem{proposition}{Proposition}
\newtheorem{lemma}{Lemma}
\newtheorem{theorem}{Theorem}
\newtheorem{corollary}{Corollary}

\newcommand{\R}{\mathbb{R}}

\newcommand{\Lim}{\mathbf{Lim}}

\newcommand{\card}{\mathop{\mathrm{card}}}

\newcommand{\imp}{\rightarrow}

\newcommand{\eqdef}{\mathrel{\mathop:}=}

\renewcommand{\epsilon}{\varepsilon}
\renewcommand{\phi}{\varphi}

\renewcommand{\leq}{\leqslant}
\renewcommand{\le}{\leqslant}
\renewcommand{\geq}{\geqslant}
\renewcommand{\ge}{\geqslant}

\newcommand\norm[1]{\left\lVert#1\right\rVert}
\newcommand{\sizednorm}[2]{#1\lVert#2#1\rVert}
\newcommand\ind{\mathbf{1}}
\newcommand{\ol}[1]{\overline{#1}}
\begin{document}

\title[\resizebox{4.5in}{!}{Rearranging absolutely convergent well-ordered series in Banach spaces}]{Rearranging absolutely convergent well-ordered series in Banach spaces}
\author[V.~Čačić]{Vedran Čačić}
\author[M.~Doko]{Marko Doko}
\author[M.~Horvat]{Marko Horvat}
\address[V.~Čačić, M.~Horvat]{Department of Mathematics, Faculty of Science, University of Zagreb, Bijenička 30, 10000 Zagreb, Croatia}
\email[V.~Čačić]{veky@math.hr}
\email[M.~Horvat]{mhorvat@math.hr}
\address[M.~Doko]{Max Planck Institute for Software Systems, Paul-Ehrlich-Str.~26, 67663 Kaiserslautern, Germany}
\email[M.~Doko]{mdoko@mpi-sws.org}


\begin{abstract}
Reordering the terms of a series is a useful mathematical device, and much is
known about when it can be done without affecting the convergence or the sum of
the series. For example, if a series of real numbers absolutely converges, we
can add the even-indexed and odd-indexed terms separately, or arrange the terms
in an infinite two-dimensional table and first compute the sum of each column. The possibility of
even more intricate reorderings prompts us to find a
general underlying principle. We identify such a principle in the setting of
Banach spaces, where we consider well-ordered series with indices beyond $\omega$, but strictly under $\omega_1$. We prove
that for every absolutely convergent well-ordered series indexed by a countable
ordinal, if the series is rearranged according to any countable ordinal,
then the absolute convergence and the sum of the series remain unchanged.
\end{abstract}

\maketitle

\section{Introduction}
During a typical mathematical upbringing, one is likely to encounter a task of adding real numbers where changing the order of summation turns out to be the trick to solving the task. The sum to be computed may involve finitely many terms, such as when proving that the sum of vertex degrees $d(v)$ in every finite undirected graph $(V,E)$ is twice the number $|E|$ of edges:
$$\sum_{v\in V}d(v)=\sum_{v\in V}\sum_{e\in E}\ind_{e}(v)=\sum_{e\in E}\sum_{v\in V}\ind_{e}(v)=\sum_{e\in E}2=2|E|.$$
It might even require adding infinitely many terms, which is known as computing the sum of a \emph{series}:
$$\frac12+\frac24+\frac38+\cdots+\frac{n}{2^n}+\cdots=$$
$$\begin{array}{ccccccccc}
=&\frac12&+&\frac14&+&\frac18&+&\cdots&+\\[4pt]
+&       & &\frac14&+&\frac18&+&\cdots&+\\[4pt]
+&       & &       & &\frac18&+&\cdots&+\\[-2pt]
+&       & &       & &       & &\ddots&=\\[-2pt]
\end{array}$$

$$=1+\frac12+\frac14+\cdots=2\;.$$
However, tricks that work in the finite case may not always be applicable when dealing with infinite series:
$$1\stackrel{?}{=}1+1+(-1)+1+(-1)+\cdots\stackrel{?}{=}2+(-1)+1+(-1)+1+\cdots\stackrel{?}{=}2.$$

The aspiring mathematics student will eventually learn when the order of summation for a given series can be changed without disturbing its convergence or its sum. For example, if a series contains only non-negative terms or, more generally, if it is absolutely convergent, then it is possible to:
\begin{itemize}
\item reorder the terms by subjecting their indices to some permutation of $\omega$,
\item split the series into an odd-indexed and an even-indexed part, sum each of the two parts separately, and add the results,
\item arrange the terms into an $\omega\times\omega$ matrix, then sum each column, and then all the results, etc.
\end{itemize}

A natural question to ask is whether all these results are special cases of some general principle. Before we give a positive answer to this question, we need to discuss our choice of \emph{setting} as well as the generalization of the notion of \emph{sequence reordering}. For the former, we choose Banach spaces as convenient ambient spaces that allow us to compute sums, take their unique limits, measure them with the defined norm, and infer convergence from absolute convergence when needed. For the latter, sequences are typically reordered by composition with some bijection from $\omega$ to $\omega$. This means that each term of the reordered sequence is preceded by finitely many other terms. In this work, we consider bijections between $\omega$ and some ordinal number $\alpha$, where ``sequence'' terms might have infinite indices. For example, the three previously mentioned reorderings are cases where $\alpha$ is respectively $\omega$, $\omega\cdot2$, and $\omega^2$.

The corresponding extension of the notion of series is known as \emph{well-ordered series}. We show that for every absolutely convergent well-ordered series indexed by a countable ordinal, if the series is rearranged according to some other countable ordinal, then the absolute convergence and the sum of the series remain unchanged.

In order to make our main contributions outlined in the previous paragraph fully precise, we provide a discussion of the related work as well as the necessary background in the following sections.

\section{Related work}
A plethora of results were being published throughout the 20th century that cover a vast range of topics at the intersection of set theory and analysis, including the absolute convergence of series in normed spaces~\cite{Dvoretzky1950,Dieudonne1960,Hildebrandt1940}. Some publications include the study of well-ordered series in various contexts~\cite{Dieudonne1960,Komjath1977,Shepherdson1950}, which even extends to recent years~\cite{Heikkila2013}. However, to the best of our knowledge, the effects of the reorderings outlined in the previous section on the absolute convergence of well-ordered series in Banach spaces have not yet been thoroughly explored.

In the book \emph{Foundations of modern analysis}~\cite{Dieudonne1960}, Dieudonné defines absolutely convergent well-ordered series $\Sigma(a_i)_{i\in\alpha}$ in Banach spaces by considering the absolute convergence of the series $\Sigma(a_{f(i)})_{i\in\omega}$ for any bijection $f\colon\omega\to\alpha$. The notion is well defined because it does not depend on the choice of bijection. This is also true for the notion of its sum, which can be lifted analogously. We choose more general definitions of well-ordered series, their convergence and sum, which do not depend on absolute convergence, and expand on the work of Dieudonné by investigating the relationship of differently indexed well-ordered series with the same terms.

\section{Background}

We define ordinals in the von Neumann sense; each ordinal is the set of all smaller ordinals. For example, $\omega$ is the set of all finite ordinals, and $\omega_1$ is the set of all countable ordinals. There are three kinds of ordinals: zero (denoted by $0$), successor ordinals and limit ordinals. We write $\Lim$ for the class of all limit ordinals (those that are neither $0$ nor successors).

A \emph{normed space} $(X,\norm{\cdot})$ is a vector space $X$ equipped with a norm $\norm{\cdot}$. We sometimes omit the norm and simply refer to $(X,\norm{\cdot})$ as $X$. If a normed space $X$ is complete, it is called a \emph{Banach space}. For an ordinal $\beta\in\omega_1\setminus\omega$ and a normed space $X$, we call a function $a\colon\beta\to X$ a \emph{hypersequence}\footnote{Similar functions are sometimes called \emph{$\beta$-sequences}~\cite{Aguilera2015} to make the index domain explicit.}. We denote the value of a hypersequence $a$ at argument (\emph{index}) $\alpha$ by $a_\alpha$. If we want to emphasize the domain, we write $(a_\alpha)_{\alpha\in\beta}$ instead of $a$. Note that we only consider countably infinite domains because of the known fact that convergent series with uncountably many positive terms contain only countably many non-zero terms~\cite{CM}.

A common way to define the sum of a series involves limits of partial sum sequences. We define the notion of \emph{well-ordered series} in an analogous fashion; we call the main building block a \emph{partial sum hypersequence}.
\begin{definition} \label{def:hypersequence}
Let $(X,\norm{\,\cdot\,})$ be a normed space. We say that the hypersequence $s\colon\gamma\to X$ is a \emph{partial sum hypersequence} of a hypersequence $a\colon\beta\to X$, if the following conditions hold:
\begin{enumerate}
\item[(1)] $1\le\gamma\le\beta+1$
\item[(2)] $s_0=0$
\item[(3)] $(\forall\alpha\in\gamma)(\alpha+1\in\gamma\Rightarrow s_{\alpha+1}=s_\alpha+a_\alpha)$
\item[(4)] $(\forall\lambda\in\gamma)(\lambda\in\Lim\Rightarrow s_\lambda=\displaystyle\lim_{\alpha\to\lambda}s_\alpha)$
\end{enumerate}
By the limit in condition \textup{(4)}, we mean
	$$(\forall\varepsilon\in\R^+)(\exists\alpha_0\in\lambda)(\forall\alpha\in\lambda)(\alpha\ge\alpha_0\Rightarrow\norm{s_\lambda-s_\alpha}<\varepsilon).$$
\end{definition}

Hypersequences are naturally partially ordered by restriction (as functions), which is the same as the subset relation when we consider functions as sets of ordered pairs.
\begin{proposition}
For every hypersequence $a$, there exists the largest partial sum hypersequence of $a$ (that is, the one of which all other partial sum hypersequences of $a$ are restrictions).
\end{proposition}

\begin{proof}
	It is possible to show, by transfinite induction, that every two partial sum hypersequences of $a$ must agree on the intersection of their domains. Namely, the conditions from the definition prescribe the values of $s$ on $0$, successors and limit ordinals in its domain, using previous values (existing limits are unique because the codomain is a normed space, hence Hausdorff).

Moreover, because their domain is bounded from above, all partial sum hypersequences of $a$ form a set---a subset of $\mathcal P\bigl((\beta+1)\times X\bigr)$. That set is also nonempty, since it contains the trivial hypersequence $\langle0\rangle$ with domain~$1$. But then the union of that set is also a partial sum hypersequence of $a$, and it is obviously the largest one.
\end{proof}

Since it is the largest, it is unique for every $a$, and we denote it by $\sigma a$. We can now give a precise definition of the notion of \emph{well-ordered series}.

\begin{definition}
Let $a$ be a hypersequence and $\sigma a$ its largest partial sum hypersequence. We define the \emph{well-ordered series} $\Sigma a$ as the ordered pair $(a,\sigma a)$.
\end{definition}

\begin{proposition}\label{domain}
For every hypersequence $a\colon\beta\to X$, the domain of $\sigma a$ is either $\beta+1$ or some limit ordinal not greater than $\beta$.
\end{proposition}

\begin{definition}
In the former case, we say that the well-ordered series $\Sigma a$ \emph{converges}, and $(\sigma a)_\beta\in X$ is called the \emph{sum} of $\Sigma a$, and denoted by $\sum_{\alpha\in\beta}a_\alpha$. In the latter case, we say that the well-ordered series $\Sigma a$ \emph{diverges}.
\end{definition}

\begin{proof}[Proof of Proposition~\cref{domain}]
We denote the domain of $\sigma a$ by $\gamma$. The first condition of Definition~\cref{def:hypersequence} says that $\gamma$ is not zero, so it is either a successor or a limit ordinal.

\begin{itemize}
\item First assume $\gamma$ is a successor, different from $\beta+1$. The first condition of Definition~\cref{def:hypersequence} says that it cannot be larger than $\beta+1$, so it must be smaller: $\gamma=\delta+1\in\beta+1$. But then $\delta\in\gamma$, which means that $(\sigma a)_\delta$ is defined, and also $\delta\in\beta$, which means $a_\delta$ is defined. Adding those two numbers, we can extend $\sigma a$ to its proper superset
$$\sigma a\cup\big\{\big(\delta+1,(\sigma a)_\delta+a_\delta\big)\big\}$$
which is also a partial sum hypersequence of $a$, contradicting $\sigma a$ being the largest one. So if $\gamma$ is a successor, it must be the successor of $\beta$.
\item Now assume $\gamma$ is a limit ordinal. The first condition of Definition~\cref{def:hypersequence} implies that it is not greater than $\beta+1$, and it cannot be equal to $\beta+1$ because it is a limit ordinal. Therefore, it must be smaller than $\beta+1$, i.e., not greater than $\beta$.\hfill\qed
\end{itemize}\let\qed\relax
\end{proof}

The concept of absolute convergence can be directly extended to well-ordered series.
\begin{definition}
We say that the well-ordered series $\Sigma(a_\alpha)_{\alpha\in\beta}$ \emph{absolutely converges} if the well-ordered series $\Sigma(\norm{a_\alpha})_{\alpha\in\beta}$ converges.
\end{definition}

It is straightforward to verify that standard series with real terms can be identified with well-ordered series over $\beta=\omega$, and the definitions of convergence and absolute convergence in that case coincide with our definitions. The domain $\gamma$ of each partial sum hypersequence $s$ of $a$ is in that case either $\omega$ (if the series diverges) or $\omega+1$ (if it converges). From this point on, we use the terms \emph{well-ordered series} and \emph{series}, and the notation $\Sigma(a_i)_{i\in\omega}$ and $\sum_{i=0}^\infty a_i$, interchangeably.

We now formulate and prove two simple lemmas that will be an often used tool in the following section. They are a direct result of the topological fact that limits commute with continuous functions, and they let us sum a series by first summing two disjoint parts, and then adding the two results together.

\begin{lemma}\label{split}
	Let $X$ be a normed space, $(a_i)_{i\in\omega}$ a sequence in $X$, and $S\subseteq\omega$. If the series $\Sigma\big(a_i\cdot\ind_S(i)\big)_{i\in\omega}$ and $\Sigma\big(a_i\cdot\ind_{\omega\setminus S}(i)\big)_{i\in\omega}$ converge, then the series $\Sigma(a_i)_{i\in\omega}$ also converges and \[\sum_{i=0}^\infty a_i=\sum_{i=0}^\infty\big(a_i\cdot\ind_S(i)\big)+\sum_{i=0}^\infty \big(a_i\cdot\ind_{\omega\setminus S}(i)\big).\]
\end{lemma}
\begin{proof}

We compute
\begin{align*}
	\sum_{i=0}^\infty a_i&=\lim_n\sum_{i=0}^n \left(a_i\cdot\big(\ind_S(i)+\ind_{\omega\setminus S}(i)\big)\right)\\&=\lim_n\left(\sum_{i=0}^n\big(a_i\cdot\ind_S(i)\big)+\sum_{i=0}^n \big(a_i\cdot\ind_{\omega\setminus S}(i)\big)\right)\\&\stackrel{(\ast)}{=}\lim_n\sum_{i=0}^n \big(a_i\cdot\ind_S(i)\big)+\lim_n\sum_{i=0}^n\big(a_i\cdot\ind_{\omega\setminus S}(i)\big)\\&=\sum_{i=0}^\infty\big(a_i\cdot\ind_S(i)\big)+\sum_{i=0}^\infty\big(a_i\cdot\ind_{\omega\setminus S}(i)\big).
\end{align*}
We know that the equality marked with $(\ast)$ holds because the addition function, $+\colon X\times X\to X$, is continuous with respect to the norm topology on $X$ and the product topology on $X\times X$, and taking limits of convergent sequences commutes with applying continuous functions.
\end{proof}

\begin{lemma}\label{abssplit}
	Let $X$ be a Banach space, $(a_i)_{i\in\omega}$ a sequence in $X$, and $S\subseteq\omega$. Then the series $\Sigma(a_i)_{i\in\omega}$ absolutely converges if and only if the series $\Sigma(a_i\cdot\ind_S(i))_{i\in\omega}$ and $\Sigma(a_i\cdot\ind_{\omega\setminus S}(i))_{i\in\omega}$ absolutely converge. Moreover, \[\sum_{i=0}^\infty a_i=\sum_{i=0}^\infty (a_i\cdot\ind_S(i))+\sum_{i=0}^\infty (a_i\cdot\ind_{\omega\setminus S}(i)).\]
\end{lemma}
\begin{proof}
Assume that the series $\Sigma(a_i)_{i\in\omega}$ absolutely converges. For all $i\in\omega$, we have $\norm{a_i}\cdot\ind_S(i)\leq\norm{a_i}$ and $\norm{a_i}\cdot\ind_{\omega\setminus S}(i)\leq\norm{a_i}$, so we conclude that the series $\Sigma\big(a_i\cdot\ind_S(i)\big)_{i\in\omega}$ and $\Sigma\big(a_i\cdot\ind_{\omega\setminus S}(i)\big)_{i\in\omega}$ absolutely converge. The converse follows by applying Lemma~\cref{split} to the sequence $(\norm{a_i})_{i\in\omega}$.

If all three series above absolutely converge, they also converge because $X$ is complete, so the desired equality follows directly from Lemma~\cref{split}.
\end{proof}

We additionally prove that the direct comparison test for series can be extended to well-ordered series.
\begin{lemma}\label{dct}
Let $X$ be a Banach space, $\alpha\in\omega_1\setminus\omega$, and $(a_i)_{i\in\alpha},(b_i)_{i\in\alpha}$ hypersequences in $X$ such that for all $i\in\alpha$, $\norm{a_i}\leq\norm{b_i}$. If the series $\Sigma (b_i)_{i\in\alpha}$ absolutely converges, then the series $\Sigma(a_i)_{i\in\alpha}$ absolutely converges and $\sum_{i\in\alpha}\norm{a_i}\leq\sum_{i\in\alpha}\norm{b_i}$.
\end{lemma}
\begin{proof}
We prove the lemma by transfinite induction. If $\alpha=\beta+1$ for some ordinal $\beta$, from the induction hypothesis we have $$\sum_{i\in\alpha}\norm{a_i}=\sum_{i\in\beta}\norm{a_i}+\norm{a_\beta}\leq\sum_{i\in\beta}\norm{b_i}+\norm{b_\beta}=\sum_{i\in\alpha}\norm{b_i}.$$ If $\alpha\in\Lim$, then for all $\beta\in\alpha$, we let $s_\beta\eqdef\sum_{i\in\beta}\norm{a_i}$ and compute $$s_\beta\leq\sum_{i\in\beta}\norm{b_i}\leq\lim_{\beta\to\alpha}\sum_{i\in\beta}\norm{b_i}=\sum_{i\in\alpha}\norm{b_i}.$$ Hence, the set of real numbers $\{s_\beta:\beta\in\alpha\}$ is bounded from above. Therefore, it has a supremum, which we denote with $L$. But then for all $\epsilon>0$, there exists an $n_\epsilon\in\omega$ such that $L-\epsilon<a_{n_\epsilon}$. For all $n\geq n_\epsilon$, we then have $L-\epsilon<s_{n_\epsilon}\leq s_n\leq L<L+\epsilon$. This implies $|s_n-L|<\epsilon$, so the hypersequence $(s_\beta)_{\beta\in\alpha}$ converges.

We also know that $\lim_{\beta\to\alpha}s_\beta\leq\sum_{i\in\alpha}\norm{b_i}$; otherwise, there would exist $i_0\in\alpha$ such that $a_{i_0}>b_{i_0}$.
\end{proof}

Now that all of the preliminaries are in place, we can analyze the effects of reorderings on the convergence and sum of well-ordered series in Banach spaces. We devote the entirety of the following section to this topic, where we focus on absolutely convergent well-ordered series.

\section{Effects of reordering on absolute convergence and sum}
In this section, we determine the effects of reorderings on the convergence and sum of absolutely convergent well-ordered series in Banach spaces. 

Previously, we proved two simple lemmas that hinted at our general proof strategy; they let us separate a series into two parts, which can be summed individually, and the results added without changing the sum. In a proof by induction, one of the parts may be covered by the induction hypothesis, and what is left might be a single number, or an insignificant quantity. The former will be the case when we look at well-ordered series up to a successor ordinal; we will encounter the latter when we consider limit ordinals, where we make use of the fact that a convergent series has arbitrarily small tail sums. We also use the fact that for an absolutely convergent series $\Sigma a_i$, tail sums of the series $\Sigma\!\norm{a_i}$ cannot be increased by leaving out some terms of $\Sigma\!\norm{a_i}$.

By employing the strategy described in the previous paragraph, we first prove that an absolutely convergent standard series in a Banach space can be rearranged into a well-ordered series with respect to any countably infinite ordinal without affecting the convergence or sum of the series. 
\begin{theorem}\label{toomega}
Let $\Sigma(a_i)_{i\in\omega}$ be an absolutely convergent series in a Banach space $X$. Then for all $\alpha\in\omega_1\setminus\omega$ and all bijections $f\colon\alpha\to\omega$, the well-ordered series $\Sigma(a_{f(i)})_{i\in\alpha}$ absolutely converges and
\[\sum_{i\in\alpha}a_{f(i)} = \sum_{i=0}^\infty a_i.\]
\end{theorem}
\begin{proof}[Proof (by transfinite induction on $\alpha$)]
If $\alpha=\omega$, the result is a basic theorem of mathematical analysis. Let $\alpha>\omega$ and let $f\colon\alpha\to\omega$ be a bijection. Suppose the statement is true for all ordinals
$\beta\in\alpha\setminus\omega$. We separately consider two cases, where $\alpha$ is respectively a successor ordinal or a limit ordinal.
\bigskip

\noindent\fbox{$\alpha=\beta+1$} We define the (hyper)sequence $(\hat{a}_i)_{i\in\omega}$ with
\[\hat{a}_i \eqdef \begin{cases}
	a_i, & \text{if }i<f(\beta)\\
	a_{i+1}, & \text{if }i \geq f(\beta).
\end{cases}\]
The function $g\colon\beta\to\omega$ defined with
\[g(\gamma) \eqdef \begin{cases}
	f(\gamma), & \text{if } f(\gamma)<f(\beta)\\
	f(\gamma)-1, & \text{if } f(\gamma)>f(\beta)
\end{cases}\]
	is a bijection with the property $a_{f(i)} = \hat{a}_{g(i)}$ for all $i\in\beta$. From Lemma~\cref{abssplit}, we can now conclude
\begin{align*}
\sum_{i=0}^{\infty}a_i &=
\sum_{i=0}^{\infty}\left(a_i\cdot\ind_{\omega\setminus\{f(\beta)\}}(i)\right)+\sum_{i=0}^{\infty}\left(a_i\cdot\ind_{\{f(\beta)\}}(i)\right) = \\
&=\sum_{i=0}^{\infty}\hat{a}_i + a_{f(\beta)}=
\sum_{i\in\beta} \hat{a}_{g(i)} + a_{f(\beta)} =
\sum_{i\in\beta} a_{f(i)} + a_{f(\beta)} =
\sum_{i\in\alpha} a_{f(i)}.
\end{align*}

\noindent\fbox{$\alpha\in\Lim$}
For $\beta\in\alpha\setminus\omega$, we define sequences $(a^\beta_i)_{i\in\omega}$ and $(\ol{a}^\beta_i)_{i\in\omega}$ with
\begin{align*}
a^\beta_i &\eqdef a_i\cdot\ind_{f[\beta]}(i)\textrm{ and}\\
\ol{a}^\beta_i &\eqdef a_i\cdot\ind_{f[\alpha\setminus\beta]}(i),
\end{align*}
where $f[\beta]$ denotes the set $\{f(\gamma):\gamma\in\beta\}$.
As in the previous case, we use Lemma~\cref{abssplit} to conclude
\begin{equation}
	\sum_{i=0}^{\infty}a_i = \sum_{i=0}^{\infty}a^\beta_i + \sum_{i=0}^{\infty}\ol{a}^\beta_i.\tag{$\ast$}
\end{equation}

For an infinite well-ordered set $(S,\leq)$ and all $n\in\omega$, we denote the set obtained by removing the first $n$ elements from $S$ with $S_n$. More precisely, we define 
\[S_0=S;\quad S_{n+1}=S_n\setminus\{\min(S_n)\}.\]
We use this notation to define a bijection $g\colon\beta\to\omega$ that assigns to each ordinal $i\in\beta$ the position of its image $f(i)$ in the well-ordered set $f[\beta]$ with
\[g \eqdef \{(i,n)\in\beta\times\omega : f(i) = \min(f[\beta]_n)\}.\]
In other words, for infinite $f[\beta]=\{f(i_0)<f(i_1)<f(i_2)<\dotsb\}\subseteq\omega$, we set ${g(i_0)\eqdef0}$, ${g(i_1)\eqdef1}$, ${g(i_2)\eqdef2}$, etc. Note that for all $i\in\beta$, $a_{f(i)} = a^\beta_{g(i)}$ and
\[ g(i) = \card \{k\in f[\beta] : k < f(i) \}. \]
From the induction hypothesis and $(\ast)$ we get
\[
\sum_{i\in\beta} a_{f(i)} =
\sum_{i\in\beta}a^\beta_{g(i)} =
\sum_{i=0}^{\infty}a^\beta_i =
\sum_{i=0}^{\infty}a_i - \sum_{i=0}^{\infty}\ol{a}^\beta_i.
\]

\begin{figure}
$$\begin{array}{cccccccccccccccc}
f[\beta] && a_0 & a_1 & a_2 & & \ldots & & a_{n_0} & & a_{n_0+2} & a_{n_0+3}&&& \ldots\\[3pt]
f[\alpha\setminus\beta] && & & & & & & & a_{n_0+1} & & & a_{n_0+4}&& \ldots 
\end{array}$$
\caption{When $\alpha$ is a limit ordinal, we split the series $\Sigma a_i$ in two parts: one with terms indexed by $f[\beta]$ that contains the ``significant'' part of the sum, and the ``small'' rest indexed by $f[\alpha\setminus\beta]$.}
\label{figsplit}
\end{figure}

The idea now is to split the series as shown in Fig.~\cref{figsplit}. To that end, let $\epsilon>0$. We are looking for an ordinal $\beta_0\in\alpha$ such that for all
$\beta\in\alpha\setminus\beta_0$,\linebreak[4] $\sizednorm\bigg{\sum\limits_{i\in\beta}a_{f(i)} -
\sum\limits_{i=0}^{\infty}a_i}<\epsilon$. Let $n_0\in \omega$ such that
$\sum_{i=n_0}^\infty\norm{a_i}<\epsilon$. We define
\[\begin{array}{r@{}r@{}l}
\beta_0 &{}\eqdef{}&
\min(\{\gamma\in\alpha : \gamma>\omega \land (\forall \gamma'\in\alpha)(\gamma'\geq\gamma \imp f(\gamma')>n_0)\})\\[1ex]
&{}={}&\max(\{\omega\}\cup f^{-1}[n_0+1])+1<\alpha.
\end{array}\]
For all $\beta\in\alpha\setminus\beta_0$, due to Lemma~\cref{dct} we have
\[
	\sizednorm\Bigg{\sum_{i\in\beta}a_{f(i)} - \sum_{i=0}^{\infty}a_i} =
	\norm{\sum_{i=0}^{\infty}\ol{a}^\beta_i} \leq
	\sum_{i=0}^{\infty}\sizednorm\big{\ol{a}^\beta_i} \leq
	\sum_{i=n_0}^\infty\norm{a_i}<\epsilon.\eqno\qed
\]\let\qed\relax
\end{proof}

Now we consider what happens when going the other way; we prove that we can safely reorder an absolutely convergent well-ordered series in a Banach space into a standard series. We will actually be able to leverage Theorem~\cref{toomega} in the proof by employing a suitable choice of reordering.

\begin{theorem}\label{fromomega}
Let $\alpha\in\omega_1\setminus\omega$ and let $\Sigma(a_i)_{i\in\alpha}$ be an absolutely convergent well-ordered series in a Banach space $X$. Then for all bijections $f\colon\omega\to\alpha$, the series $\Sigma(a_{f(i)})_{i\in\omega}$ absolutely converges and
\[\sum_{i=0}^\infty a_{f(i)} = \sum_{i\in\alpha} a_i.\]
\end{theorem}
\begin{proof}
For every $n\in\omega$, we denote the $n$-th partial sum of the series $\Sigma(\sizednorm{}{a_{f(i)}})_{i\in\omega}$ with $s_n$ and compute $$s_n=\sum_{i\in\omega}\sizednorm{}{a_{f(i)}}\cdot\ind_n(i)=\sum_{i\in\alpha}\norm{a_i}\cdot\ind_n(f^{-1}(i))\leq\sum_{i\in\alpha}\norm{a_i},$$ where the last inequality can be obtained directly from Lemma~\cref{dct}. Therefore, $(s_n)_{n\in\omega}$ is a monotonically increasing sequence of real numbers that is bounded from above, so it converges.

To prove the second part of the theorem, let $\hat{a}_i\eqdef a_{f(i)}$ for all $i\in\omega$. We know $f^{-1}\colon\alpha\to\omega$ is a bijection, so Theorem~\cref{toomega} gives us
	\[ \sum_{i=0}^\infty a_{f(i)}=\sum_{i=0}^\infty \hat{a}_i=\sum_{i\in\alpha}\hat{a}_{f^{-1}(i)}=\sum_{i\in\alpha} a_{f(f^{-1}(i))} = \sum_{i\in\alpha} a_i.\eqno\qed\]\let\qed\relax
\end{proof}

The two theorems can now be combined to prove that any reordering, i.e., a bijection between countably infinite ordinals, of an absolutely convergent well-ordered series in a Banach space yields another absolutely convergent well-ordered series where the initial sum is preserved.
\begin{corollary}
Let $\alpha,\beta \in\omega_1\setminus\omega$ and let $\Sigma(a_i)_{i\in\alpha}$ be an absolutely convergent well-ordered series. Then for all bijections $f\colon\beta\to\alpha$, the well-ordered series $\Sigma(a_{f(i)})_{i\in\beta}$ absolutely converges and 
\[ \sum_{i\in\beta} a_{f(i)} = \sum_{i\in\alpha} a_i \,. \]
\end{corollary}
\begin{proof}
	Let $g\colon\omega\to\alpha$ be a bijection. Then $(g^{-1}\circ f)\colon\beta\to\omega$ is also a bijection. If we define $\hat{a}_i\eqdef a_{g(i)}$ for all $i\in\omega$, from Theorems~\ref{toomega}~and~\ref{fromomega} we get
	\[\sum_{i\in\alpha}a_i=\sum_{i=0}^\infty a_{g(i)}=\sum_{i=0}^\infty \hat{a}_i=\sum_{i\in\beta} \hat{a}_{(g^{-1}\circ f)(i)}=\sum_{i\in\beta}a_{f(i)}.\eqno\qed\]\let\qed\relax
\end{proof}

\section{Conclusion}
In this paper, we wanted to take a step towards a general treatise of the
effects of reorderings on the convergence and sum of well-ordered series. Our
main focus was on reordering absolutely convergent well-ordered series in
Banach spaces. Such reorderings often occur in mathematical practice, so
identifying a general underlying principle is of practical as well as
theoretical interest.

To define the notion of well-ordered series itself, we followed the modern way
of employing transfinite recursion over the structure of the index set. Our
chosen setting of Banach spaces may be possible to improve upon when absolute
convergence is replaced by a different kind of convergence; in the future,
we would like to extend our analysis to more general uniform spaces, such
as Hausdorff topological abelian groups, and investigate in detail the effects
of reorderings on the conditional and unconditional convergence of well-ordered
series in such a setting.

\bigskip
\acknowledgements{We would like to thank Zvonimir Šikić for introducing the problem of rearranging well-ordered series to us. We also thank Siniša Miličić, Filip Nikšić, Zvonimir Sviben, Domagoj Vrgoč, and Mladen Vuković for helpful comments and discussions.}

\bibliographystyle{abbrv}

\end{document}